\documentclass[a4paper,12pt]{article}

\usepackage{a4,amssymb,amsmath,amsthm,url,xcolor,enumerate,float,lineno}
\usepackage{bezier,amsfonts,amssymb,graphicx,amsthm,url}
\usepackage[utf8]{inputenc}
\usepackage{subcaption}
\usepackage{amsmath}
\usepackage{amsfonts}
\usepackage{amssymb,amsthm}
\usepackage{url}
\usepackage{hyperref}
\usepackage{tikzit}
\usepackage{xcolor,colortbl}

\newtheorem{theorem}{Theorem}[section]

\newtheorem{proposition}[theorem]{Proposition}

\newtheorem{lemma}[theorem]{Lemma}
\newtheorem*{theorem*}{Theorem}
\newtheorem*{problem*}{Problem}

\def\kaxxa{{\vcenter {\hrule height .2mm
\hbox{\vrule width .2mm height 2mm \kern 2mm
\vrule width .2mm} \hrule height .2mm}}}

\newcommand{\cay}[2]{\mathsf{Cay}\left(#1 ; #2\right)}

\renewcommand{\mod}[2]{#1 \ \left(\mathrm{mod} \ #2\right)}

\tikzstyle{vertex}=[fill=black, draw=black, shape=circle, thick, scale=0.75]
\tikzstyle{blue_clique}=[fill=white, draw=blue, shape=circle, ultra thick]
\tikzstyle{clique}=[fill=white, draw=black, shape=circle, scale=0.75]
\tikzstyle{red_vertex}=[fill=red, draw=red, shape=circle, scale=0.75, thick]

\tikzstyle{edge}=[-, fill={rgb,255: red,0; green,128; blue,128}]
\tikzstyle{dir_edge}=[->]
\tikzstyle{blue_edge}=[-, draw=blue, ultra thick]
\tikzstyle{red_edge}=[-, draw=red, ultra thick]
\tikzstyle{box}=[-, fill=gray, fill opacity=0.2]
\tikzstyle{dashed_edge}=[-, dashed]

\tikzset{every picture/.style={font issue=\footnotesize},
         font issue/.style={execute at begin picture={#1\selectfont}}
        }

\title{On $(r,c)$-constant, planar and circulant graphs}
\author{\begin{tabular}{cc}Yair Caro & Xandru Mifsud \\ University of Haifa-Oranim & University of Malta \\ \href{mailto:yacaro@kvgeva.org.il}{\small yacaro@kvgeva.org.il} & \href{mailto:xmif0001@um.edu.mt}{\small xmif0001@um.edu.mt}\end{tabular}}
\date{}

\DeclareGraphicsExtensions{.pdf,.png,.jpg}

\begin{document}

\maketitle

\begin{abstract}
This paper concerns $(r,c)$-constant graphs, which are $r$-regular graphs in which the subgraph induced by the open neighbourhood of every vertex has precisely $c$ edges. The family of $(r,c)$-graphs contains vertex-transitive graphs (and in particular Cayley graphs), graphs with constant link (sometimes called \textit{locally isomorphic graphs}), $(r,b)$-regular graphs,  strongly regular graphs, and much more. 

This family was recently introduced in \cite{caro2023} serving as important tool in constructing flip graphs \cite{caro2023, mifsud2024}.
 
In this paper we shall mainly deals with the following:

\begin{enumerate}[i.]
	\item Existence and non-existence of $(r, c)$-planar graphs.  We completely determine the cases of existence and non-existence of such graphs and supply the smallest order in the case when they exist.
 
 	\item We consider the existence of $(r, c)$-circulant graphs. We prove that for $c \equiv \mod{2}{3}$ no $(r,c)$-circulant graph exists and that for $c \equiv \mod{0, 1}{3}$, $c > 0$ and $r \geq 6 + \sqrt{\frac{8c - 5}{3}}$ there exists $(r,c)$-circulant graphs. Moreover for $c = 0$ and $r \geq 1$, $(r, 0)$-circulants exist.
 
	\item We consider the existence and non-existence of small $(r,c)$-constant graphs,  supplying a complete table of the smallest order of graphs we found for $0 \leq c \leq \binom{r}{2}$ and $r \leq 6$. We shall also determine all the cases in this range for which $(r,c)$-constant graphs don’t exist. We establish a public database of $(r,c)$-constant graphs for varying $r$, $c$ and order.
\end{enumerate}
\end{abstract}
\section{Introduction}

This paper concerns $(r, c)$-constant graphs (or simply $(r,c)$-graphs) which are $r$-regular graphs in which the subgraph induced by the open neighbourhood of every vertex has precisely $c$ edges. 

The family of $(r,c)$-graphs contains within it a large number of other families of important graphs, such as vertex-transitive graphs \cite{Dobson_Malnic_Marusic_2022} (and in particular Cayley graphs \cite{LauriScap} and circulants \cite{yu2023classification}), graphs with constant link \cite{brouwer2023locally}, $(r,b)$-regular graphs \cite{conder2021parameters}, strongly regular graphs \cite{Brouwer_Van_Maldeghem_2022}, and more. The containment of these families is illustrated in Figure \ref{RC_Const_Hierarchy}.

\begin{figure}[ht!]
\centering
\includegraphics[width=0.68\textwidth]{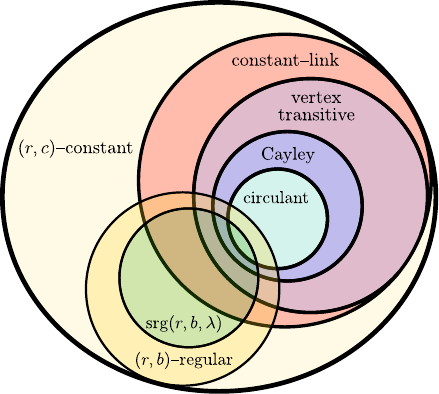}
\caption{Hierarchy of $(r, c)$-constant graphs and their sub-families of interest.}
\label{RC_Const_Hierarchy}
\end{figure} 

The family of $(r, c)$-graphs was recently introduced in \cite{caro2023}, serving as an important tool in the construction of flip-graphs.  We mention that an equivalent notion to $(r,c)$-graphs, with a focus and motivation on self-complementary graphs and strongly regular graphs, was introduced in \cite{NAIR1996201} (after preliminary work in \cite{RADHAKRISHNANNAIR1994205}) under the name `strongly vertex triangle regular graphs'. 

The main facts concerning $(r,c)$-graphs in \cite{caro2023} are summarised below:
\begin{enumerate}
	\item[Fact 1] If $G$ and $H$ are, respectively, an $(r_1, c_1)$-graph and an $(r_2, c_2)$-graph then their Cartesian product $G \ \square \ H$ is an $(r_1 + r_2, c_1 + c_2)$-graph. 
	\item[Fact 2] For $r \geq 1$ and $0 \leq c \leq \frac{r^2}{2} - 5r^{\frac{3}{2}}$, there exists an $(r,c)$-graph.
	\item[Fact 3] Furthermore, for $k \geq 1$ and $r \geq 3k$ there does not exist an $\left(r,\binom{r}{2} - k\right)$-graph. 
	\item[Fact 4] Using Corollary 2.5 in \cite{RADHAKRISHNANNAIR1994205} (and observing that $t(v)$ in the notation of \cite{RADHAKRISHNANNAIR1994205} is $e(v)$ in our notation), it follows that if $G$ is an $(r,c)$-graph on $n$ vertices then the complement $\overline{G}$ is an $\left(n-1-r, \binom{n-1}{2} -  \frac{3r(n-1-r)}{2} - c \right)$-graph.  
\end{enumerate}

The \textit{spectrum} of $r$, denoted by $\mathsf{spec}(r)$, is the set of all such integers $c$ such that an $(r,c)$-graph exists. Therefore from the above, the integer interval $\left[0, \frac{r^2}{2} - 5r^{\frac{3}{2}}\right]$ is contained in  $\mathsf{spec}(r)$. One can also define the \textit{spectrum} of $c$, denoted by $\mathit{spec}(c)$, which is the set of all integers $r$ such that an $(r,c)$-graph exists. We note that $\mathit{spec}(c)$ is in fact determined by the smallest $r$ for which an $(r, c)$-graph $G$ exists, as by Fact 1 above we have that $G \ \square \ K_2$ is an $(r + 1, c)$-graph.

However the order of the constructed graphs by the cartesian product has been shown to be large with respect to the smallest known cases constructed by other methods, in particular using circulants.

From Fact 2 we have that for fixed $c$, for $r \geq \sqrt{2c} \big(1 + o(1)\big)$ there always exists an $(r,c)$-graph. However the proof involves once more the Cartesian product of graphs with constant-links, where the links are the vertex disjoint union of complete graphs. In particular, these complete graphs are determined by solving the Diophantine equations $\sum x_j = r$ and $\sum \binom{x_j}{2} = c$.  

Hence the constructed $(r,c)$-graphs are themselves graphs with constant-links, but the use of the Cartesian products results in the actual construction producing very large graphs. 

By Fact 3 we establish that if $r \geq 3k$ and $\binom{r}{2} - k \leq c \leq \binom{r}{2} - 1$, then no $(r,c)$-graph exists. However, there still remains open the problem of determining for which $c$, where $\frac{r^2}{2} - 5r^{\frac{3}{2}} \leq c \leq \binom{r}{2} - \frac{r}{3}$, such that $(r,c)$-graphs exist. 

Taking into account the fact that Cartesian products do not preserve properties such as $\mathcal{F}$-minor free, where $\mathcal{F}$ is a family of graphs, it is also of interest to consider families of $\mathcal{F}$-minor free graphs and determining which $(r,c)$-graphs belong to these families.

In light of the facts and explanation above, the paper is organised as follows:

In Section 2 we present our results concerning $(r,c)$-planar graphs,  showing in particular that the only pairs $(r, c)$ for which $(r,c)$-planar graphs exist are: $(1,0), (2,0), (2,1), (3,0), (3,1), (3,3), (4,1), (4,2), (4,3), (4,4), (5,4)$ and  $(5,5)$. We also exhibit the smallest such graphs and prove the non-existence of the remaining cases.
 
In Section 3 we consider the existence of $(r, c)$-circulants. Somewhat surprisingly, we prove that $(r,c)$-circulants with $c \equiv \mod{2}{3}$ do not exist. On the other-hand we constructively prove that for given $c \equiv \mod{0,1}{3}$, $c > 0$ and $r \geq 6 + \sqrt{\frac{8c - 5}{3}}$ there exists $(r,c)$-circulant graphs. Moreover for $c = 0$ and $r \geq 1$, $(r, 0)$-circulants exist.  
 
The focus on circulants is motivated by the fact that they are among the simplest structured family of graphs which completely lie within the families of vertex transitive, constant-link, and Cayley graphs. They also play an important role in many other branches of graph theory (such as lower-bound construction for Ramsey numbers, as well as colouring and independence numbers in locally sparse graphs \cite{anderson2024coloring}) .    
 
In Section 4 we provide a complete answer to the question for which pairs $(r,c)$, where $r \leq 6$ and $0 \leq c \leq \binom{r}{2}$, do $(r,c)$-graphs exist. In each case of existence, we supply the smallest example we have found. Otherwise, the other pairs are proved not to be realised by $(r, c)$-graphs. Working on Section 4, we have introduced a database of $(r,c)$-graphs which can be found at \cite{r_c_database}.
 
Lastly, in Section 5 we offer several open problems emerging out of this work.

Before proceeding any further, we introduce some notation. By $n$ we denote the number of vertices in a graph, whilst by $e$ we denote the number of edges. In the case of a planar graph, let $f$ denote the number of faces. 

By $\delta, \Delta$ and $d$ we shall denote the smallest, largest and average degree of a graph $G$, respectively. By $g$ we will denote the girth of the graph.

Given a vertex $v$ in a graph $G$, by $N(v)$ and $N[v]$ we denote the open and closed neighbourhoods of $v$ in $G$, respectively. Furthermore, by $e(v)$ we denote the number of edges in the subgraph induced by the vertices in the open neighbourhood $N(v)$.
\section{Existence of $(r,c)$--planar graphs}

In this section we completely solve the existence problem of $(r,c)$-planar graphs. Since every planar graph $G$ has $\delta(G) \leq 5$ it suffices to consider $1 \leq r \leq 5$ and  $0 \leq c \leq \binom{r}{2}$. Table \ref{planar_rc_values} summarises our findings, where for every possible pair of $(r, c)$ values, \textsc{ne} denotes that an $(r, c)$-planar graph does not exist. Otherwise, we give the number of vertices of the smallest known example. 

In the case of existence we shall provide the witness example, whilst in the case of non--existence we shall give the necessary proofs. 

\begin{table}[h!]
\[\begin{array}{c|ccccccccccc}
\tikz{\node[below left, inner sep=1pt] (def) {$r$};%
      \node[above right,inner sep=1pt] (abc) {$c$};%
      \draw (def.north west|-abc.north west) -- (def.south east-|abc.south east);}
 & 0 & 1 & 2 & 3 & 4 & 5 & 6 & 7 & 8 & 9 & 10\\
\hline
1 & 2 & \cellcolor{lightgray} & \cellcolor{lightgray} & \cellcolor{lightgray} & \cellcolor{lightgray} & \cellcolor{lightgray} & \cellcolor{lightgray} & \cellcolor{lightgray} & \cellcolor{lightgray} & \cellcolor{lightgray} & \cellcolor{lightgray}\\
2 & 4 & 3 & \cellcolor{lightgray} & \cellcolor{lightgray} & \cellcolor{lightgray} & \cellcolor{lightgray} & \cellcolor{lightgray} & \cellcolor{lightgray} & \cellcolor{lightgray} & \cellcolor{lightgray} & \cellcolor{lightgray}\\
3 & 8 & 6 & \textsc{ne} & 4 & \cellcolor{lightgray} & \cellcolor{lightgray} & \cellcolor{lightgray} & \cellcolor{lightgray} & \cellcolor{lightgray} & \cellcolor{lightgray} & \cellcolor{lightgray} \\
4 & \textsc{ne} & 24 & 12 & 8 & 6 & \textsc{ne} & \textsc{ne} & \cellcolor{lightgray} &  \cellcolor{lightgray} & \cellcolor{lightgray} & \cellcolor{lightgray} \\
5 & \textsc{ne} & \textsc{ne} & \textsc{ne} & \textsc{ne} & 24 & 12 & \textsc{ne} & \textsc{ne} & \textsc{ne} & \textsc{ne} & \textsc{ne}
\end{array}\]
\caption{Existence and non-existence of $(r, c)$-planar graphs, with the size of the smallest known construction listed in the case of existence.}
\label{planar_rc_values}
\end{table}%

\vspace{-5mm}

\subsection{Smallest known $(r,c)$-planar graphs}

We briefly summarise in Table \ref{planar_rc_smallest} the smallest existing $(r, c)$-planar graphs.

\begin{table}[h!]
\[\begin{array}{c|cccccc}
\tikz{\node[below left, inner sep=1pt] (def) {$r$};%
      \node[above right,inner sep=1pt] (abc) {$c$};%
      \draw (def.north west|-abc.north west) -- (def.south east-|abc.south east);}
 & 0 & 1 & 2 & 3 & 4 & 5 \\
\hline
1 & K_2 & \cellcolor{lightgray} & \cellcolor{lightgray} & \cellcolor{lightgray} & \cellcolor{lightgray} & \cellcolor{lightgray}\\
2 & C_4 & K_3 & \cellcolor{lightgray} & \cellcolor{lightgray} & \cellcolor{lightgray} & \cellcolor{lightgray}\\
3 & Q_3 & C_3 \ \square \ K_2 & \cellcolor{lightgray} & K_4 & \cellcolor{lightgray} & \cellcolor{lightgray} \\
4 & \cellcolor{lightgray} & \textrm{Figure \ref{planar_4_1}} & \textrm{Figure \ref{planar_4_2}} & 4\textrm{-antiprism} & 3\textrm{-antiprism} & \cellcolor{lightgray} \\
5 & \cellcolor{lightgray} & \cellcolor{lightgray} & \cellcolor{lightgray} & \cellcolor{lightgray} & \textrm{Figure \ref{planar_5_4}} & \textrm{Figure \ref{planar_5_5}}
\end{array}\]
\caption{Smallest known $(r,c)$-planar graphs.}
\label{planar_rc_smallest}
\end{table}

\vspace*{-4mm}

\begin{figure}[h!]
	\begin{minipage}{0.5\textwidth}
		\centering
		\includegraphics[width=0.85\textwidth]{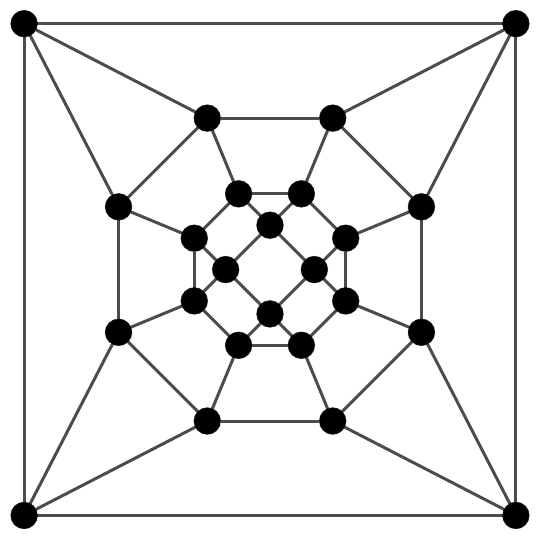}
		\caption{$(4,1)$-planar}
		\label{planar_4_1}
	\end{minipage}%
	\begin{minipage}{0.5\textwidth}
		\centering
		\includegraphics[width=0.85\textwidth]{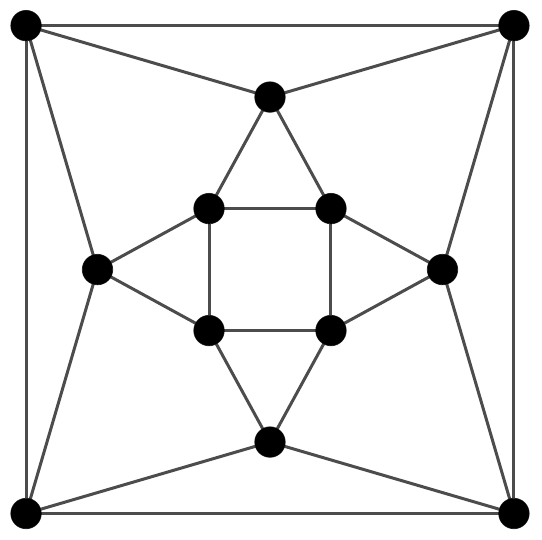}
		\caption{$(4,2)$-planar}
		\label{planar_4_2}
	\end{minipage}
\end{figure}

\begin{figure}[h!]
	\begin{minipage}{0.5\textwidth}
		\centering
		\includegraphics[width=0.85\textwidth]{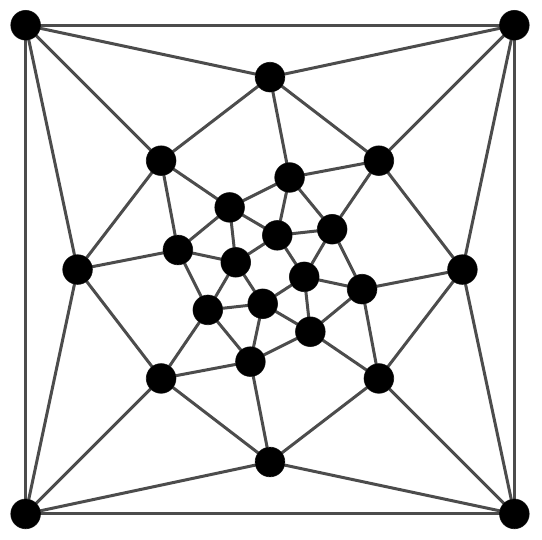}
		\caption{$(5,4)$-planar}
		\label{planar_5_4}
	\end{minipage}
	\begin{minipage}{0.5\textwidth}
		\centering
		\includegraphics[width=0.85\textwidth]{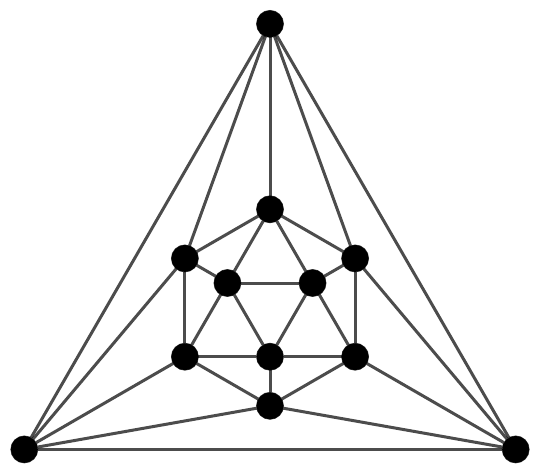}
		\caption{$(5,5)$-planar}
		\label{planar_5_5}
	\end{minipage}
\end{figure}

\subsection{Proofs for the cases of non--existence}

Several cases of non-existence are easy to prove. The cases when $(r, c)$ is $(3, 2)$, $(4, 5)$, or $(5, 9)$ are ruled out by Fact 3 in the introduction as they are of the form  $c = \binom{r}{2} - 1$ with $r \geq 3$.
 
The cases $(4, 6)$ and $(5, 10)$ are realised only by $K_5$ and $K_6$ respectively, which are non-planar. The case $(5, 8)$ (and similarly $(5,9)$ and $(5, 10)$) cannot be realised by planar graphs as for a vertex $v$, the graph induced on $N[v]$ has 6 vertices but at least 13 edges. This contradicts the fact that in a planar graph on $n$ vertices there are at most $3n - 6$ edges.
 
On the other hand, the case of $(5, 7)$ is not possible because the graph induced on $N[v]$ has 6 vertices and 12 edges and a vertex of degree 5, but none of the four graphs on 6 vertices, 12 edges and $\Delta = 5$ are planar \cite{harary}.
 
The remaining cases are treated below. We observe that proving the theorem below supplies, in particular, a proof of the non-existence case of the remaining $(r, c)$ cases.

\begin{theorem}\label{planarNE}
No planar graph $G$ exists with
\begin{enumerate}[i.]
	\item $d(G) \geq 4$  and $e(v) = 0$ for all vertices $v$   
	\item $d(G) = 5$  and $e(v) = c$ for all vertices $v$, $c \in \{1, 2, 3\}$  
	\item $d(G) \leq 5$  and $e(v) = 6$ for all vertices $v$
\end{enumerate}
\end{theorem}

Before proceeding with the proof of Theorem \ref{planarNE}, we note the following classical results on planar graphs \cite{west}.

\begin{proposition}\label{planar_1}
	Let $f_j$ denote the number of faces on $j$ edges in a planar graph. Then $2e = \sum_{j \geq 3} j (f_j)$.
\end{proposition}

\begin{proposition}[Euler's Polyhedral Formula] \label{planar_2}
	For a planar graph, $n - e + f = 2$.\end{proposition}
      
\begin{proposition} \label{planar_3}
	For a planar graph with girth $g$, $e \leq \frac{g(n-2)}{g-2}$.
\end{proposition}

\begin{proof}[Proof of Theorem \ref{planarNE}]
	Suppose that $G$ is a planar graph with $d(G) \geq 4$ and $e(v) = 0$ for all vertices $v$. Then $G$ must be triangle free and hence $g \geq 4$. By Proposition \ref{planar_3}, $e \leq \frac{4(n-2)}{2} = 2n-4$ and therefore $d(G) \leq \frac{2e}{n} \leq \frac{4n-8}{n} < 4$ which is a contradiction. Hence (i) follows.
	
	Now suppose that $d(G) = 5$ and $e(v) = c$ for all vertices $v$, $c \in \{1, 2, 3\}$. Let $f = f_3 + f_{\geq 4}$, where $f_{\geq 4}$ denotes the number of faces of on $\geq 4$ edges. We get by Proposition \ref{planar_1} that $2e \geq  3f_3 +4f_{\geq 4}$. Consider the bipartite graph having partite sets $A$ and $B$, where $A$ is the set of vertices of $G$ and $B$ is the set of triangles in $G$. Every vertex in $A$ is adjacent to $c$ vertices in $B$, and every vertex in $B$ is adjacent to three vertices in $A$. Therefore counting the number of edges between $A$ and $B$, we get that $c n = c |A| = 3 |B| = 3 f_3$. Hence $f_3 = \frac{c n}{3}$.
	
	Now from Euler's Polyhedral Formula we have that
	\begin{align*}
		8 &= 4(n - e + f)& \\
		&= 4(n-e) + f_3 + (3f_3 + 4f_{\geq 4})& \\
		&\leq 4(n-e) + \frac{cn}{3} + 2e& \because f_3 = \frac{c n}{3},  2e \geq  3f_3 +4f_{\geq 4} \\
		&=\frac{(12 + c)n}{3} -2e &
	\end{align*}
	and hence $(12 + c)n \geq 6e + 24 \geq 3n d(G) + 24$. Re-arranging, we get $d(G) \leq 4 + \frac{c}{3} - \frac{8}{n}$. Consequently, since $1 \leq c \leq 3$, we have that $d(G) < 5$ which is a contradiction. Therefore (ii) follows.
	
	Lastly, suppose that $d(G) \leq 5$  and $e(v) = 6$ for all vertices $v$. By a similar argument for $f_3$ in (ii), we obtain that $f_3 = 2n$ in this case. Therefore $$2 = n - e + f = n - e + f_3 + f_{\geq 4} \geq 3n - e$$
	and hence $6n \leq 2e + 4 \leq n d(G) + 4$. Re-arranging, we get $d(G) \geq 6 - \frac{4}{n} > 5$, which is a contradiction. Hence (iii) follows, completing the proof.
\end{proof}
 
\section{Existence and construction of $(r,c)$-circulant graphs}

Recall that every circulant is an $(r,c)$-graph, hence by $(r,c)$-circulant we mean a circulant graph which is $r$-regular such that every vertex $v \in V$ satisfies $e(v) = c$. 

We have seen in \cite{mifsud2024} that suitable coloured $(r,c)$-circulants were useful to obtain small $2$-flip graphs. Also, extensive computer searches for small $(r,c)$-graphs reveals that in many cases the smallest order of an $(r,c)$-graph is realised by a circulant.

Another motivation is that, parallel to the determination of $\mathsf{spec}(r)$ and $\mathit{spec}(c)$, it is of interest to consider these spectrums restricted to circulant, namely $\mathsf{spec}_{\vert\mathrm{circ}} (r)$ and $\mathit{spec}_{\vert\mathrm{circ}} (c)$. 


\subsection{Existence of $(r,c)$-circulants}

Consider the circulant $\mathsf{Circ}(n, S)$ where $S$ is the set of
\textit{jumps} $i \in S$ where $1 \leq i \leq \frac{n}{2}$. Then this
circulant corresponds to the Cayley graph $\cay{\mathbb{Z}_n}{S
\cup -S}$. We shall study the open neighbourhood of a vertex in
$\mathsf{Circ}(n, S)$. 

Our first result is that, somewhat unexpectedly, there exists no $(r,c)$-circulants for $c \equiv \mod{2}{3}$ and hence for such values of $c$ we have $\mathit{spec}_{\vert\mathrm{circ}} (c) = \emptyset$. A warm-up for this result is the following simple proposition.

\begin{proposition}
	If $G$ is an $(r,c)$-graph on $n \equiv \mod{1, 2}{3}$ vertices, then $c \equiv \mod{0}{3}$.
\end{proposition}

\begin{proof}
Let $t(G)$ denote the number of triangles in $G$. Every vertex $v$ has $e(v) = c$ and hence $v$ is a vertex on exactly $c$ triangles. 	Counting (with multiplicities) over all vertices, we get $cn$ triangles. On the other hand, every triangle is counted this way three times, hence this double counting gives $t(G) = \frac{cn}{3}$. But as $n \equiv \mod{1, 2}{3}$ it follows that $c \equiv \mod{0}{3}$. 
\end{proof}

The case when $n \equiv \mod{0}{3}$ is summarised in the following theorem and requires a more involved argument. 

\begin{theorem}\label{circ_existence}
	For any vertex $v$ in $\mathsf{Circ}(n, S)$, $e(v) \equiv \mod{0}{3}$ except when $n \equiv \mod{0}{3}$ and $\frac{n}{3} \in S$, in which case $e(v) = 1$.	
\end{theorem}

As a consequence of this theorem, it follows that there exists no $(r,c)$-circulant such that $c \equiv \mod{2}{3}$. Before proceeding with the proof of this theorem, we require the introduction of a canonical way of writing edges in a circulant. Observe that any edge in $\mathsf{Circ}(n, S)$ can be written as $\{ x, x + y \}$ for some $x \in S \cup - S$ and $y \in S$, where in the case that $y$ is an involution we require that $x \in S$. Let this be a canonical way of writing the edges in a circulant. We show that any edge is uniquely expressed in this manner.

\begin{lemma}
  Every edge $e$ in $\mathsf{Circ}(n, S)$ has a unique canonical representation.
\end{lemma}

\begin{proof}
  Consider an edge $e$ in $\mathsf{Circ}(n, S)$. Let $x, a \in S \cup - S$ and $y, b \in S$ such that $e = \{ x, x + y\} = \{ a, a + b \}$ are two distinct canonical representations of $e$. Then it must be that $a = x + y$ and therefore $b = - y$, as otherwise if $a = x$ then they are not distinct canonical representations.
  
  Since $y \in S$ then $- y \in - S$, but $- y = b \in S$. Hence $y$ must be
  an involution and therefore $b = y$. But in this case, either $x \notin S$
  or $a \notin S$ and one of these representations is not canonical. Therefore $e$ has a unique canonical representation. 
\end{proof}

Consider a canonically represented edge $e = \{ x, x + y \}$. The \textit{orbit} $\langle e \rangle$ of an edge $e$ is the set of edges $\{ a, b \}$ such that $\{ a, b, b - a \} \subseteq \{ x, - x, y, - y, x + y, - x - y\}$. This corresponds to \textit{at most} 6 different edges: $\{x, x + y \}, \{ y, x + y \}, \{ x, - y \}, \{ y, - x \}, \{ - x, - x - y \}$ and $\{ - y, - x - y \}$. We note the following useful lemma.

\begin{lemma}\label{canonical_eq_classes}
  Let $e_1, e_2$ be two edges in $\mathsf{Circ}(n, S)$. If $e_1 \in \langle e_2 \rangle$ then $\langle e_1 \rangle = \langle e_2 \rangle$.
\end{lemma}

\begin{proof}
  Let $e_1$ and $e_2$ have the canonical representations $\{ x, x + y \}$ and $\{ a, a + b \}$, respectively, for $x, a \in S \cup - S$ and $y, b \in S$. Since $e_1 \in \langle e_2 \rangle$ then $a, b, a + b \in \{ x, - x, y, - y, x + y, - x - y \}$. Then, considering the subgraph induced between $a, b$ and $a + b$ and their inverses, we get that this is the same as the subgraph induced on $x, y, x + y$ and their inverses. Since $\langle e_1 \rangle$ and $\langle e_2 \rangle$ are defined as the edges in these respective subgraphs, it follows that $\langle e_1 \rangle = \langle e_2 \rangle$.
\end{proof}

This leads to a partition into equivalence classes of the edges in the open neighbourhood of a vertex. We exemplify this for the subgraph induced by
$N (0)$ in $\mathsf{Circ}(12, \{ 1, 3, 4, 6\})$ illustrated in Figure \ref{c_mod_3}.

\begin{figure}[h!]
	\centering
		\includegraphics[width=0.3\textwidth]{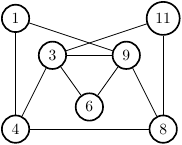}
		\caption{Subgraph induced by $N (0)$ in $\mathsf{Circ}(12, \{ 1, 3, 4, 6\})$.}
		\label{c_mod_3}
\end{figure}

The edges in this case are partitioned as follows:
\begin{enumerate}
  \item ${\langle}\{1,4\}{\rangle} = \{\{1,4\}, \{3,4\}, \{1,9\}, \{3,11\},
  \{8,11\}, \{8,9\}\}$ since $\{ 1, 4 \}$ has the canonical representation $\{ x, x
  + y \}$ with $x = 1$, $y = 3$ and $x + y$ = 4, and therefore $- x = 11, - y
  = 9$ and $- x - y = 8$.
  
  \item $\langle \{ 3, 6 \} \rangle = \{ \{ 3, 6 \}, \{ 3, 9 \}, \{ 6, 9 \}
  \}$ since $\{ 3, 6 \}$ has the canonical representation $\{ x, x + y \}$ with $x
  = y = 3$ and $x + y = 6$, and therefore $- x = - y = 9$ and $- x - y = 6$.
  
  \item $\langle \{ 4, 8 \} \rangle = \{ \{ 4, 8 \} \}$ since $\{ 4, 8 \}$
  has the canonical representation $\{ x, x + y \}$ with $x = y = 4$ and $x + y =
  8$, and therefore $- x = - y = 8$ and $- x - y = 4$.
\end{enumerate}

We are now in a position to prove our main result for this section. 

\begin{proof}[Proof of Theorem \ref{circ_existence}]
	Realising that a circulant $\mathsf{Circ}(n, S)$ is the Cayley graph $\cay{\mathbb{Z}_n}{S \cup - S}$, by vertex-transitivity we will only consider $e(0)$.
  
  	Consider then the edge $\{ x, x + y \}$ where $x \in S \cup - S$ and $y \in S$, where $x \neq - y$. Six cases may arise:
  	\begin{enumerate}[i.]
    	\item If $x \neq y$, $x$ and $y$ not involutions, $2 x \neq - y$ and $2 y \neq - x$, then $x, y, - x, - y, x + y$ and $- x - y$ are all distinct neighbours of 0 with the following distinct edges in $\langle \{ x, x + y\} \rangle$: $\{ x, x + y \}, \{ y, x + y \}, \{ x, - y \}, \{ y, - x \}, \{ - x, - x - y \}$ and $\{ - y, - x - y \}$.
    
    	\item If $x \neq y, x$ and $y$ not involutions and $2 x = - y$, then $x, y, - x, - y$ are all distinct neighbours of $0$ with the following distinct edges in $\langle \{ x, x + y \} \rangle$: $\{ x, - x \}, \{ y, -x \}$ and $\{ x, - y \}$. The case when $2 y = - x$ follows similarly; note that it cannot be that $2 x = - y$ and $2 y = - x$, as otherwise $x = y = 0$.
    
    	\item If $x \neq y$, $x$ is an involution and $2 y \neq x$, then $x, y, -y, x + y$ and $x - y$ are all distinct neighbours of $0$ with the following distinct edges in $\langle \{ x, x + y \} \rangle$: $\{ x, x + y\}, \{ y, x + y \}, \{ x, - y \}, \{ y, x \}, \{ x, x - y \}$ and $\{ - y, x - y \}$. The case when $y$ is an involution and $2 x \neq y$ follows similarly.
    
    	\item If $x \neq y$, $x$ is an involution and $2 y = x$, then $x, y$, and $-y$ are all distinct neighbours of $0$ with the following edges in $\langle \{ x, x + y \} \rangle$: $\{ x, - y \}, \{ y, x\}$ and $\{ y, - y \} .$ The case when $y$ is an involution and $2 x = y$ follows similarly.
    
    	\item If $x = y$ and $x \neq \frac{n}{3}$ then $x, - x$, $2 x$ and $- 2 x$ are distinct neighbours of 0 with three distinct edges in $\langle \{ x, x+ y \} \rangle$: $\{ x, 2 x \}$, $\{ - x, - 2 x \}$ and $\{ x, - x \}$.
    
    	\item If $x = y$ and $x = \frac{n}{3}$ then $2 x = - x$ and $\{ x, - x \}$ is the only edge in $\langle \{ x, x + y \} \rangle$.
  	\end{enumerate}
  
  	By Lemma \ref{canonical_eq_classes} and the cases above, it follows that the subgraph induced by the open neighbourhood $N (0)$ has an edge set which can be partitioned into sets of size 6, 3 or 1. Observe that there is \textit{exactly} one case leading to a partition of size 1, namely when we $\mod{n}{3} = 0$, $\frac{n}{3} \in S$ and therefore $\left\{ \frac{n}{3}, - \frac{n}{3} \right\}$is an edge in the open neighbourhood of $0$. The result follows.
\end{proof}

We now show that for the remaining cases when $c \equiv \mod{0, 1}{3}$, $c > 0$ and $r \geq 6 + \sqrt{\frac{8c - 5}{3}}$, then an $(r,c)$-circulant exists. In the case when $c = 0$, we show that $(r,0)$-circulants exist for $r \geq 1$. We first require the following notation and lemma. 

Let $k, j, l$ be integers such that $k > j \geq 0$ and $l \geq 0$. Define the set $S_{k,j} = \{1, \dots, k - 1, k + j\}$. Let $R_{k,j,l}$ be the set containing the first $l$ terms of the arithmetic progression starting at $2(k+j)+1$ and with different $d = k + j + 1$. Define the set $S_{k,j,l} = S_{k,j} \cup R_{k,j,l}$. For convenience, also define $S_{k,j,0} = S_{k,j}$. 

\begin{lemma} \label{circulant_nhood_lemma}
	Let $n, k, j, l$ be integers such that $n > k > j \geq 0$ and $l \geq 0$. Then, for every vertex $v$ in $\mathsf{Circ}\left(n, S_{k, j, l}\right)$,
	\begin{enumerate}[i.]
		\item If $l = 0$ and $n = 3k$, $e(v) = 1 + 3 \binom{k-1}{2} + 3(k-1-j)$.
		\item Otherwise, given $n$ sufficiently large, $e(v) = 3 \binom{k-1}{2} + 3(k-1-j)$.
	\end{enumerate}
\end{lemma}

\begin{proof}
	We first show that, for $n \in \mathbb{N}$ such that $n > 3k$, the circulant $\mathsf{Circ}\left(n, S_{k, j,0}\right)$ has $3 \binom{k-1}{2} + 3(k-1-j)$ edges in every open neighbourhood.
	
	It suffices to consider the open neighbourhood of vertex $0$. The vertices $\{1, \dots, k-1\}$ and $\{-1, \dots, -k+1\}$ are neighbours of $0$ and respectively induce a $(k-1)$-clique, contributing $2\binom{k-1}{2}$ edges to $N(0)$. More so, given $1 \leq i < k$, the vertex $-k + i$ has $k - 1 - i$ neighbours in $\{1, \dots, k-1\}$, contributing another $\binom{k-1}{2}$ edges to $N(0)$. By a similar argument, there are $3(k-1-j)$ edges in $N(0)$ which are incident to either $k + j$ or $-k - j$. Hence $e(0) = 3\binom{k-1}{2} + 3(k - 1 - j)$. Hence the second case for $l = 0$ follows.
	
	We now show the second case for $l \geq 1$. Recall that $S_{k,j,l}$ is the set $S_{k,j,0}$ with an arithmetic progression of length $l$ and difference $d = k+j+1$ added to it. Given $n$ sufficiently large, and since the difference $d$ is greater than the largest value in $S_{k,j,0}$, the addition of this arithmetic progression does not change the number of edges in the open neighbourhoods, since there are no intersecting sums between these sets and their inverses. Hence $e(0) = 3\binom{k-1}{2} + 3(k - 1 - j)$ in $\mathsf{Circ}\left(n, S_{k, j, l}\right)$.
	
	Recall that the case $e(0) \equiv \mod{1}{3}$ for a circulant $\mathsf{Circ}\left(n, S\right)$ is only possible if, and only if, $\mod{n}{3} \equiv 0$ and $\frac{n}{3} \in S$. Since $k \in S_{k, j,0}$, considering the case when $n = 3k$ we get an additional edge in every open neighbourhood and therefore the first case follows.
\end{proof}

The existence of $(r,c)$-circulants for $c \equiv \mod{0}{3}$ follows almost immediately from this lemma.

\begin{proposition}\label{prop_3_6}
	Let $k, c, r$ be integers such that $k \geq 1$, $r \geq 2k$, $c \equiv \mod{0}{3}$ and $\frac{3(k-2)(k-1)}{2} \leq c \leq \frac{3k(k-1)}{2}$. Then there exists an $(r,c)$-circulant.
\end{proposition}

\begin{proof}
	Let $j$ be an integer such that $0 \leq j < k$ and $c = 3\binom{k-1}{2} + 3(k - 1 - j)$. Suppose there exists $l \in \mathbb{Z}$ such that $r = 2(k + l)$. For sufficiently large $n$, by Lemma \ref{circulant_nhood_lemma} the circulant $\mathsf{Circ}\left(n, S_{k, j, l}\right)$ is an $(r,c)$-circulant.
	
	Otherwise if $r = 2(k+l) + 1$, choose $n$ sufficiently large such that no sum from $S_{k, j, l}$ is equal to $\frac{n}{2}$. Then $\mathsf{Circ}\left(n, S_{k, j, l} \cup \left\{\frac{n}{2}\right\}\right)$ is an $(r,c)$-circulant, since $n$ is such that the addition of $\frac{n}{2}$ does not introduce any new edges in the open neighbourhoods. The result follows.
\end{proof}

We now consider the case when $c \equiv \mod{1}{3}$, summarised in the following two propositions for $r \equiv \mod{0}{2}$ and $r \equiv \mod{1}{2}$ respectively.

\begin{proposition}
	Let $k, c, r$ be integers such that $k \geq 1$, $r \geq 2k$, $r \equiv \mod{0}{2}$, $c \equiv \mod{1}{3}$ and $\frac{3(k-2)(k-1)}{2} \leq c \leq \frac{3k(k-1)}{2}$. Then there exists an $(r,c)$-circulant.
\end{proposition}

\begin{proof}
	Let $j$ be an integer such that $0 \leq j < k$ and $c - 1 = 3\binom{k-1}{2} + 3(k - 1 - j)$. Let $l$ be an integer such that $r = 2(k + l)$. The case when $l = 0$ follows immediately by Lemma \ref{circulant_nhood_lemma} (i), considering the circulant $\mathsf{Circ}\left(3k, S_{k, j, 0}\right)$. 
	
	Consider the case when $l \geq 1$. Let $n$ be such that $n \equiv \mod{0}{3}$ and $\frac{n}{3} > \max\left(S_{k,j,l-1}\right)$. In particular let $n$ be sufficiently large such by Lemma \ref{circulant_nhood_lemma} the circulant $\mathsf{Circ}\left(n, S_{k, j, l - 1}\right)$ is an $(r - 2, c - 1)$-circulant.
	
	Since $\frac{n}{3} > \max\left(S_{k,j,l-1}\right)$, the addition of $\frac{n}{3}$ to $S_{k, j, l - 1}$ increases the degree by 2 but only adds a single edge to every open neighbourhood. 
	
	Therefore $\mathsf{Circ}\left(n, S_{k, j, l - 1} \cup \left\{\frac{n}{3}\right\}\right)$ is an $(r, c)$-circulant, as required.
\end{proof}

\begin{proposition}
	Let $k, c, r$ be integers such that $k \geq 1$, $r \geq 2k + 3$, $r \equiv \mod{1}{2}$, $c \equiv \mod{1}{3}$ and $\frac{3(k-2)(k-1)}{2} \leq c \leq \frac{3k(k-1)}{2}$. Then there exists an $(r,c)$-circulant.
\end{proposition}

\begin{proof}
	Let $j$ be an integer such that $0 \leq j < k$ and $c - 1 = 3\binom{k-1}{2} + 3(k - 1 - j)$. Let $l$ be an integer, $l \geq 1$, such that $r = 2(k + l) + 1$.
		
	Furthermore, let $n$ be such that $n \equiv \mod{0}{6}$ and $\frac{n}{6} > \max\left(S_{k,j,l-1}\right)$. In particular let $n$ be sufficiently large such by Lemma \ref{circulant_nhood_lemma} the circulant $\mathsf{Circ}\left(n, S_{k, j, l - 1}\right)$ is an $(r - 3, c - 1)$-circulant.
	
	Since $\frac{n}{6} > \max\left(S_{k,j,l-1}\right)$, the addition of $\frac{n}{3}$ to $S_{k, j, l - 1}$ increases the degree by 2 but only adds a single edge to every open neighbourhood. 
	
	Moreover, by our choice of $n$, no sum from $S_{k,j,l-1} \cup \left\{\frac{n}{3}\right\}$ is equal to $\frac{n}{2}$. Therefore the addition of $\frac{n}{2}$ only increases the degree by $1$ but does not increase the number of edges in the open neighbourhoods.
	
	Therefore $\mathsf{Circ}\left(n, S_{k, j, l - 1} \cup \left\{\frac{n}{3}, \frac{n}{2}\right\}\right)$ is an $(r, c)$-circulant, as required.
\end{proof}

Combining the above results and applying simple algebra we find that the construction above implies that $(r,c)$-circulants exist for $c \equiv \mod{0, 1}{3}$, $c > 0$ and $r \geq 6 + \sqrt{\frac{8c - 5}{3}}$. The case of equality is realised by the construction for the case $r = 2k + 3$ and  $c = 3\binom{k-1}{2} + 1$. Also, the fact that $(r,0)$-circulants exists for $r \geq 1$ comes from Proposition \ref{prop_3_6} (for $r \geq 2$) and by the trivial circulant $\mathsf{Circ}\left(n, \left\{\frac{n}{2}\right\}\right)$ for $n \geq 2$ even in the case when $r = 1$.

\section{Small $(r,c)$-graphs}

In this section we completely solve the existence problem of $(r,c)$-graphs for $1 \leq r \leq 6$ and $0 \leq c \leq \binom{r}{2}$. In Table \ref{rc_values} we summarise our findings, writing \textsc{NE} whenever an $(r,c)$-graph does not exist and supplying the number of vertices of the smallest known example otherwise.


\begin{table}[h!]
\[\begin{array}{c|cccccccccccccccc}
\tikz{\node[below left, inner sep=1pt] (def) {$r$};%
      \node[above right,inner sep=1pt] (abc) {$c$};%
      \draw (def.north west|-abc.north west) -- (def.south east-|abc.south east);}
 & 0 & 1 & 2 & 3 & 4 & 5 & 6 & 7 & 8 & 9 & 10 & 11 & 12 & 13 & 14 & 15\\
\hline
1 & 2 & \cellcolor{lightgray} & \cellcolor{lightgray} & \cellcolor{lightgray} & \cellcolor{lightgray} & \cellcolor{lightgray} & \cellcolor{lightgray} & \cellcolor{lightgray} & \cellcolor{lightgray} & \cellcolor{lightgray} & \cellcolor{lightgray} & \cellcolor{lightgray} & \cellcolor{lightgray} & \cellcolor{lightgray} & \cellcolor{lightgray} & \cellcolor{lightgray}\\
2 & 4 & 3 & \cellcolor{lightgray} & \cellcolor{lightgray} & \cellcolor{lightgray} & \cellcolor{lightgray} & \cellcolor{lightgray} & \cellcolor{lightgray} & \cellcolor{lightgray} & \cellcolor{lightgray} & \cellcolor{lightgray} & \cellcolor{lightgray} & \cellcolor{lightgray} & \cellcolor{lightgray} & \cellcolor{lightgray} & \cellcolor{lightgray}\\
3 & 6 & 6 & \textsc{ne} & 4 & \cellcolor{lightgray} & \cellcolor{lightgray} & \cellcolor{lightgray} & \cellcolor{lightgray} & \cellcolor{lightgray} & \cellcolor{lightgray} & \cellcolor{lightgray} & \cellcolor{lightgray} & \cellcolor{lightgray} & \cellcolor{lightgray} & \cellcolor{lightgray} & \cellcolor{lightgray}\\
4 & 8 & 9 & 9 & 7 & 6 & \textsc{ne} & 5 & \cellcolor{lightgray} &  \cellcolor{lightgray} & \cellcolor{lightgray} & \cellcolor{lightgray} & \cellcolor{lightgray} & \cellcolor{lightgray} & \cellcolor{lightgray} & \cellcolor{lightgray} & \cellcolor{lightgray}\\
5 & 10 & 12 & 12 & 10 & 12 & 12 & 8 & \textsc{ne} & \textsc{ne} & \textsc{ne} & 6 & \cellcolor{lightgray} & \cellcolor{lightgray} & \cellcolor{lightgray} & \cellcolor{lightgray} & \cellcolor{lightgray} \\
6 & 12 & 15 & 15 & 13 & 12 & 12 & 11 & 12 & 12 & 9 & 9 & \textsc{ne} & 8 & \textsc{ne} & \textsc{ne} & 7
\end{array}\]
\caption{Existence and non-existence of $(r, c)$-graphs for $1 \leq r \leq 6$, with the size of the smallest known construction listed in the case of existence.}
\label{rc_values}
\end{table}%

The smallest known examples were found through a combination of constructions and computer searches. The cases for non-existence are proven in the following sub-section. A database with these examples and many more for $r \geq 7$ has been compiled and is publicly available at \cite{r_c_database}.

\subsection{Non-existence of $(r,c)$-graphs for $1 \leq r \leq 6$}

In this section we prove the non-existence cases in Table \ref{rc_values}. From Fact 3, a number of these cases are immediately ruled out, as summarised below.

\begin{proposition}
	No $(r,c)$-graph exists for $(r,c) = (3,2), (4,5), (5,9), (6,13),$ and $(6,14)$.
\end{proposition}

\begin{proof}
	Consider $(r,c) \in \left\{(6, 13), (6, 14)\right\}$. Then by Fact 3 with $r = 6$ and $k = 2$, for $c \in \{13, 14\}$ no $(6, c)$-graph is possible. The other cases follow by a similar argument.
\end{proof}

Such an argument does not work for the remaining three cases, namely $(r,c) = (5,7), (5,8),$ and $(6,11)$. Instead, we prove non-existence by considering all possible graphs on $r$ vertices and $c$ edges, and illustrate if an $(r,c)$-graph were to exist, no open neighbourhood is isomorphic to any of the possible candidate graphs.

\begin{proposition}
	No $(r,c)$-graph exists for $(r,c) = (5,7), (5,8),$ and $(6,11)$.
\end{proposition}

\begin{proof}
	Suppose that there exists a graph $G$ which is a $(5,8)$-graph. Then every open neighbourhood in $G$ is isomorphic to a graph with $5$ vertices and $8$ edges. There are only two such graphs (\cite{harary}, p. 217), namely $K_1 + C_4$ and $\overline{P_3 \cup 2K_1}$. 

	\begin{figure}[h!]
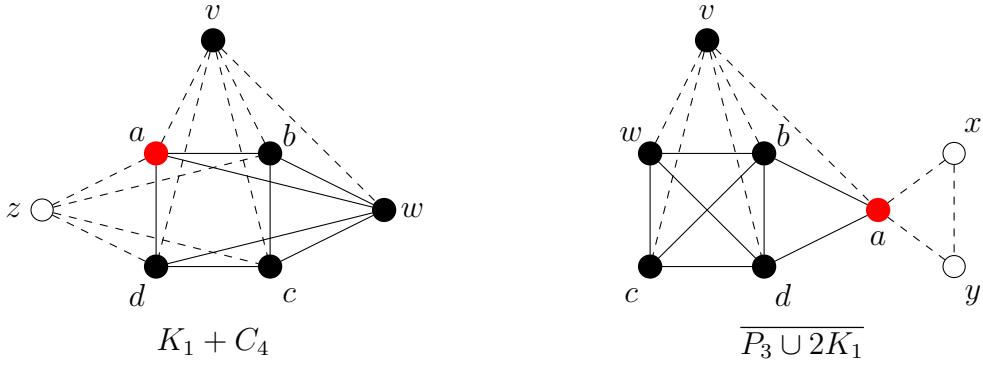

		\centering
		\ctikzfig{figures/bvl_1}
		\caption{Illustration of a `bad' vertex $a$ in the open neighbourhood $\{a,b,c,d,w\}$ of a vertex $v$ in a (hypothetical) $(5,8)$-graph. Both possible cases are shown, corresponding to the two graphs on $5$ vertices and $8$ edges.}
		\label{bvl_1}
	\end{figure}
	
	Consider a vertex $v$ in $G$ such that $N(v)$ is isomorphic to $K_1 + C_4$. Let $a, b, c, d$ and $w$ be the vertices in $N(v)$, such that $w$ is adjacent to $a, b, c$ and $d$. Therefore $v$ and $w$ have degree $5$ in the closed neighbourhood, whilst $a, b, c$ and $d$ have degree $4$ and consequently must have some other neighbour not in $N[v]$.
	
	Without loss of generality, consider vertex $a$ and let $z$ be the neighbour of $a$ not in $N[v]$. Then the open neighbourhood of $a$ is $N(a) = \{z, v, w, b, d\}$, where $v, w, b$ and $d$ have $5$ edges between them. Note that $z$ can be adjacent to $b$ and $d$, but not to $v$ and $w$ since they have degree $5$ in $N[v]$. Consequently, $z$ contributes at most two edges to $N(a)$ and therefore the open neighbourhood of $a$ has no more than $7$ edges in $G$. But $G$ is a $(5,8)$-graph, and therefore we have a contradiction.
	
	Otherwise, suppose that $N(v)$ is isomorphic to $\overline{P_3 \cup 2K_1}$, such that the vertices $b, c, d, w$ induce a $4$-clique in the open neighbourhood, whilst $a$ is only adjacent to $b$ and $d$. Then in the closed neighbourhood $N[v]$, the vertices $b, c, d$ and $w$ all have degree $5$, but vertex $a$ has degree $3$ and therefore must have two other neighbours $x$ and $y$ which are not in $N[v]$.
	
	The open neighbourhood of $a$ is $N(a) = \{v, b, d, x, y\}$, where $v, b$ and $d$ have $3$ edges between them. Note that $x$ and $y$ may possibly have an edge between them, but cannot be adjacent to $v, b$ or $d$ (as otherwise they would have degree greater than $5$). Therefore the open neighbourhood of $a$ has no more than $4$ edges in $G$, obtaining a contradiction once more.
		
	These two cases are illustrated in Figure \ref{bvl_1}. It follows that such a graph cannot exist. The cases when $(r,c) = (5,7)$ and $(6, 11)$ follow similarly, with $4$ and $9$ possibilities respectively (\cite{harary} p. 217, p. 223), illustrated in Figure \ref{bvl_5_8}.
\end{proof}

\begin{figure}[h!]
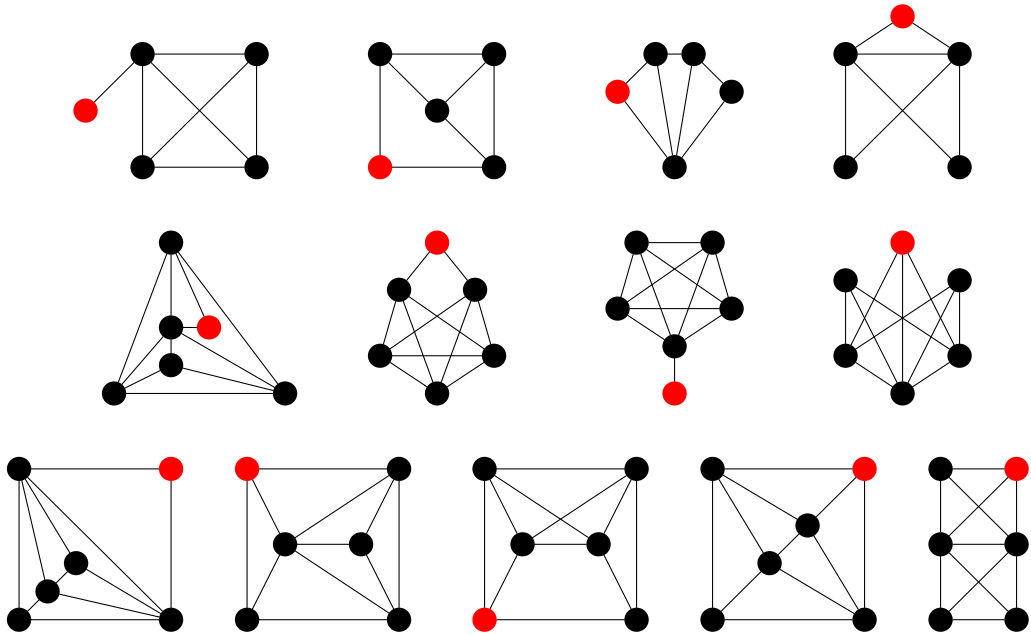

		\centering
		\ctikzfig{figures/bvl_5_8}
		\vspace*{-1.5cm}
		\caption{Illustration of a `bad' vertex in all possible open neighbourhoods for a (hypothetical) $(5,7)$-graph and $(6,11)$-graph.}
		\label{bvl_5_8}
\end{figure}

\section{Concluding remarks}

We conclude by mentioning a number of open problems, as well as providing some remarks on the compilation of our $(r,c)$-graph database.

\subsection{Open problems}

We have completely solved the existence problem for $(r,c)$-planar graphs. Still of interest is the existence of $(r,c)$-graphs in other graph families, such as 1-planar graphs \cite{FABRICI2007854}. Another existence problem of interest is that of $k$-colourable $(r,c)$-graphs, namely to determine $\mathrm{spec}(k, r) = \left\{c \colon \exists \ \mbox{$(r,c)$-graph with $\chi(G) \leq k$}\right\}$.

We have also established a database concerning $(r,c)$-graphs, and it would be of interest to extends the results of Section 4 to completely determine the existence and non-existence of $(r,c)$-graphs for $r = 7, 8$.

The determination of $h(r,c) = \min\left\{|G| \colon  \mbox{$G$ is an $(r,c)$-graph} \right\}$ and $g(r,c) =$ $\min\left\{|G| \colon \mbox{$G$ is an $(r,c)$-circulant}\right\}$ are also problems of interest. This is largely due to the role these graphs play in other problems (such as Ramsey-type problems, flip-colourings of graphs, etc).

Our results in Section 3 reveal that for a fixed $c = \mod{0,1}{3}$,  $g(r,c)$ grows linearly with $r$.  Likewise, our results in Section 4 reveal that $h(r,c)$ grows linearly in $r$ for $c = \mod{2}{3}$ as  well. 

Certainly another important problem is to determine for which $c$,  $\frac{r^2}{2}  - 5 r^\frac{3}{2} \leq c \leq \binom{r}{2}$, there exists $(r,c)$-graphs.

There are a number of other different directions worth exploring. Recently there has been growing interest in the opposite direction, where $e(v)$ is distinct for all vertices \cite{stevanovic2024searching}, as well as in a number of adjacent areas \cite{link_reg}.

\subsection{Searching for $(r,c)$-graphs}

The $(r,c)$-graph database we have compiled at \cite{r_c_database} has been generated from a number of different sources and computational techniques. Many large instances were found from the \verb|GraphData| repository \cite{graphData}. A substantial number of small instances were generated using \verb|geng| \cite{MCKAY201494} as well as \verb|plantri| \cite{397541} for $(r,c)$-planar graphs. Presently, the database contains 1794 distinct $(r,c)$-graphs for $c > 0$, with $r$ ranging from $2$ to $776$. Additionally, there are 1887 non-bipartite $(r,0)$-graphs, as well as 1007 planar $(r,c)$-graphs (51 of which have $c > 0$). In the case that a graph has a constant-link, we also give what it is. We also supply any comments from the \verb|GraphData| repository, if available.

Not all graphs we have generated are in the database (for practical reasons), however we include in \cite{r_c_database} the necessary extensions to \verb|geng| which allow for the efficient generation of $(r,c)$-graphs of a given order $n$.

In \verb|geng|, graphs of a given order $n$ are constructed by adding vertices one by one. A graph generated at some intermediate step has the property that it is a vertex-induced subgraph of any graph generated subsequently. Suppose that at some intermediate step the graph constructed has some vertex with degree greater than $r$ or an open neighbourhood with more than $c$ edges. Then clearly continuing to add vertices to this graph will not result in an $(r,c)$-graph on $n$ vertices, and therefore we can `prune' the search space by not considering this candidate graph any further. 

We achieve this by making use of the \verb|PRUNE| and \verb|PREPRUNE| preprocessor variables for \verb|geng|, in order to define appropriate functions which do such pruning. A number of other considerations are also made in order to find $(r,c)$-graphs effectively.

\section*{Acknowledgments}

The authors thank Brendan McKay for a number of insightful remarks into the implementations of \verb|geng| and \verb|plantri|.

\bibliographystyle{plain}
\bibliography{r_c_graphs_bibliography}

\end{document}